\documentclass[10pt, english,oneside,reqno]{amsart}

\usepackage[utf8]{inputenc}
\usepackage[english]{babel}
\usepackage{hyperref}
\usepackage{color}
\usepackage{enumerate}
\usepackage{fancyhdr}

\usepackage{amsmath}
\usepackage{amscd}
\usepackage{amsfonts}
\usepackage{amsthm}
\usepackage{amstext}
\usepackage{amssymb}
\usepackage{amsbsy}
\usepackage{mathrsfs}
\usepackage{tikz-cd}

\usepackage{float}
\usepackage{caption}
\usepackage{subcaption}
\usepackage{graphicx}

\usepackage[alphabetic,abbrev]{amsrefs}

\usepackage{url}

\theoremstyle{plain}
\newtheorem{theorem}{Theorem}[section]

\newtheorem{definition}[theorem]{Definition}
\newtheorem{proposition}[theorem]{Proposition}
\newtheorem{corollary}[theorem]{Corollary}
\newtheorem{lemma}[theorem]{Lemma}

\theoremstyle{remark}

\numberwithin{equation}{section}
\numberwithin{figure}{section}

\newcommand{\eps}{\varepsilon}

\newcommand{\Z}{\mathbb{Z}}

\newcommand{\R}{\mathbb{R}}
\newcommand{\C}{\mathbb{C}}

\newcommand{\Ad}{\mathrm{Ad}}

\newcommand{\Haarof}[1]{m_{#1}}

\newcommand{\supp}[1]{\mathrm{supp}(#1)}

\begin{document}
	\title[Effective Density of Random Walks on Homogeneous Spaces]{Effective Density of Non-Degenerate Random Walks on Homogeneous Spaces}
	\author[wk]{Wooyeon Kim}
	\author[ck]{Constantin Kogler}
	\thanks{The first author was supported by the Korea Foundation for Advanced Studies (KFAS). The second author gratefully acknowledges support from the European Research Council (ERC) grant No. 803711 as well as from the CCIMI at Cambridge.}
	
	\email{wooyeon.kim@math.ethz.ch, kogler@maths.ox.ac.uk}
	
	\address{Wooyeom Kim, Department of Mathematics, ETH Zürich, Rämistrasse 101, 8092 Zürich, Switzerland}
	
	\address{Constantin Kogler, Mathematical Institute, University of Oxford, Radcliffe Observatory Quarter, Woodstock Road, Oxford OX2 6GG, United Kingdom}
	
	\begin{abstract}
		We prove effective density of random walks on homogeneous spaces, assuming that the underlying measure is supported on matrices generating a dense subgroup and having algebraic entries. The main novelty is an argument passing from high dimension to effective equidistribution in the setting of random walks on homogeneous spaces, exploiting the spectral gap of the associated convolution operator.  
	\end{abstract}
	\maketitle
	
	\tableofcontents
	
	\section{Introduction}
	
	The goal of this paper is to establish effective density of certain countably supported random walks on homogeneous spaces of general simple Lie groups. 
	
	Let $G$ be a connected simple Lie group, $\Lambda < G$ a lattice, and $X = G/\Lambda$. We consider the height function  $\mathrm{ht}: X \to \mathbb{R}_{\geq 1}$ on $X$ as defined and discussed in Section~\ref{SectionQuantitativeNonDivergence}. For the purposes of the introduction, we mention that the height of a point $x_0 \in X$ measures how deep $x_0$ is in any of the cusps of $X$. Moreover, if $X$ is compact then  $\mathrm{ht} \equiv 1$. For any $\mathbf{h} > 0$ we denote by $X(\mathbf{h})$ the set of $x \in X$ with $\mathrm{ht}(x) \leq \mathbf{h}$.
	
	For a subset $S \subset G$ and $x_0 \in X$, we define $$ \mathrm{diam}_{r}(X,S,x_0) = \min\{ \ell \geq 0 \,:\, S^{\ell}x_{0} \text{ is } r\text{-dense in }  X(r^{-1}) \},$$ where we say that the set $S^{\ell}x_0$ is $r$-dense in $ X(r^{-1})$ if for every $y \in  X(r^{-1})$ there is $x \in S^{\ell}x_0$ such that $d_X(x,y) < r$, for $d_X$ the metric on $X$ defined in Section~\ref{SectionOutline}.

	Our first result is an estimate of $\mathrm{diam}_{r}(X,S,x_0)$ under the assumption that $S \subset G$ is a symmetric subset (i.e. $S = S^{-1}$) supported on matrices with algebraic entries and generating a dense subgroup. To formulate our result, denote by $\mathfrak{g}$ the Lie algebra of $G$ and by $\mathrm{Ad} : G \to \mathrm{GL}(\mathfrak{g})$ the adjoint representation. For the asymptotic notation used, we refer to Section~\ref{Notation}.
	
	\begin{theorem}\label{EffectiveDiameter}
		Let $G$ be a connected simple Lie group with finite center, $\Lambda < G$ a lattice and $X = G/\Lambda$. Let $S \subset G$ be a symmetric set generating a dense subgroup of $G$. Assume further that there is a basis of $\mathfrak{g}$ such that $\mathrm{Ad}(S)$ consists of matrices with algebraic entries with respect to the chosen basis of $\mathfrak{g}$. Then for $x_0 \in X$ and $r > 0$,
		\begin{equation}\label{EffectiveDiameterEstimate}
			\mathrm{diam}_{r}(X,S,x_0) \ll_{\Lambda,S} \log r^{-1} + \log \mathrm{ht}(x_0),
		\end{equation}
		where the implied constant depends on $\Lambda$ and $S$. 
	\end{theorem}
	
	Concrete examples of subsets $S$ in $G$ satisfying the assumptions of Theorem~\ref{EffectiveDiameter} can be constructed by considering an Iwasawa decomposition of $G = KAN$ or by choosing finitely many one parameter unipotent subgroups that generate $G$.
	
	We note that Theorem~\ref{EffectiveDiameter} is new even in the case $G = \mathrm{SL}_2(\R)$ and is not known to the authors for any Zariski dense countable set $S$. On the other hand, for $X = \mathrm{SL}_d(\R)/\mathrm{SL}_d(\Z)$ and certain $S$ generating a solvable subgroup of $\mathrm{SL}_d(\R)$ and arising from a rational iterated function system, Theorem~\ref{EffectiveDiameter} follows for $x_0 = e\Gamma \in X$ from the recent effective equidistribution result of Khalil-Luethi \cite{KhalilLuethi2022}. Moreover, if $X$ is compact, then Theorem~\ref{EffectiveDiameter} follows from the quantitative density result on $G$ as established in section 9 of \cite{DeroinHurtado2020}.
	
	The implied constant of \eqref{EffectiveDiameterEstimate} is complicated to compute and depends on the spectral gap (defined in \eqref{gapdef}), Diophantine properties and further growth parameters of the set $S$. Nonetheless, the asymptotic behaviour being logarithmic in $r^{-1}$ is optimal. 
	
	For compact groups, Theorem~\ref{EffectiveDiameter} follows from the spectral gap result established by the Bourgain-Gamburd method \cite{BourgainGamburd2008Invent} as developed for general compact simple Lie groups by Benoist-de Saxcé \cite{BenoistDeSaxce2016}. With current techniques, it is necessary to assume that the entries of the matrices in $S$ are algebraic. Indeed, without this assumption the corresponding result is not even known for compact groups. For arbitrary sets $S$ in compact groups that generate a dense subgroup, a poly-logarithmic rate for the $r$-diameter of $S$ follows from the Solovay-Kitaev algorithm. As was shown in \cite{SolovayKitaevNonCompact} and \cite{Kuperberg2015}, the Solovay-Kitaev algorithm extends to non-compact perfect Lie groups groups, which leads to a poly-logarithmic estimate of $\mathrm{diam}_{r}(X,S,x_0)$ for arbitrary finite subsets $S \subset \mathrm{SL}_d(\C)$ that generate a dense subgroup. The exponent in the poly-logarithmic rate was improved in   \cite{BoulandGiurgicaTiron2021} and \cite{Kuperberg2023}. For all these examples, it is believed that a logarithmic diameter bound holds. 
	
	Recall that a Zariski dense subgroup of $G$ is either dense or discrete. In the case when $\Gamma = \langle S \rangle$ is discrete, further difficulties arise as the orbit of $\Gamma x_0$ may be finite. If $\Gamma$ is moreover a lattice in $G$, the density of the $\Gamma x_0$ orbit in $G/\Lambda$ can be understood by studying the $\Delta(G)(x_0,e\Gamma)$-orbit on the homogeneous space $G/\Lambda \times G/\Gamma$, where $\Delta: G\to G\times G$ is the diagonal embedding. For the latter case, if one assumes $G = \mathrm{SL}_2(\R)$ and that $\Lambda$ and $\Gamma$ are arithmetic lattices, it appears that one may apply the recent results of Lindenstrauss-Mohammadi \cite{LindenstraussMohammadi2022} to deduce effective density of the $\Gamma$ orbit at $x_0$. 
	
	Let $\mu$ be a probability measure on $G$  whose support is a finite symmetric subset generating a dense subgroup. Then the recent result of Bénard \cite{Benard2021}, using the landmark measure classification theorem of Benoist-Quint \cite{BenoistQuint2013}, implies that for all $x_0 \in X$,
	\begin{equation}\label{ConvolutionEquidistribution}
		\mu^{*n} * \delta_{x_0} \to \Haarof{X}
	\end{equation}
	as $n$ tends to infinity, for $\Haarof{X}$ the Haar probability measure on $X$. Under the additional assumption that $\mathrm{Ad}(\mathrm{supp}(\mu))$ consists of matrices with algebraic entries with respect to a basis of $\mathfrak{g}$, Theorem~\ref{EffectiveDiameter} gives an effective estimate of how dense the support of $\mu^{*n}*\delta_{x_0}$ is in $X$. 
	
	Let $Z_1, Z_2, \ldots $ be independent $\mu$-distributed random variables on $G$. For $x_0 \in X$ denote by $$Y_{n,x_0} = Z_n \cdots Z_1 x_0.$$ The sequence $(Z_1, Z_2, \ldots)$ is distributed according to the probability measure $\mu^{\otimes \mathbb{N}}$, yielding a probability distribution of the random sequence $(Y_{n,x_0})_{n \geq 1}$. A further result by \cite{BenoistQuint2013} states that $\mu^{\otimes \mathbb{N}}$-almost surely the orbit $(Y_{n,x_0})_{n \geq 1}$ equidistributes, i.e. for all $f \in C_c(X)$, 
	\begin{equation}\label{OrbitEquidistribution}
		\lim_{N \to \infty}\frac{1}{N}\sum_{n = 0}^{N-1} f(Y_{n,x_0}) = \int f \, d\Haarof{X}.  
	\end{equation} Our second result is an effective density theorem for the orbit $(Y_{n,x_0})_{n \geq 1}$.
	
	\begin{theorem}\label{EffectiveDensity}
		Let $G$, $\Lambda$, $X$ and $S$ be as in Theorem~\ref{EffectiveDiameter} and let $\mu$ be a probability measure on $G$ with support $S$. Then for $A > 0$ large enough depending on $\Lambda$ and $\mu$ the following holds: For any $x_0 \in X$,  
		\begin{equation}\label{ProbEstimateEffectiveDensity}
			\mathbb{P}[(Y_{1,x_0}, \ldots , Y_{\lceil r^{-A}\rceil,x_0}) \text{ is not } r\text{-dense in } X(r^{-1})] \leq \mathrm{ht}(x_0)\cdot r^{\alpha \cdot A},
		\end{equation} for $r$ small enough in terms of $\Lambda, \mu$ and $A$ and for $\alpha = \alpha(\Lambda,\mu) > 0$ a constant depending on $\Lambda$  and $\mu$. Moreover, $\mu^{\otimes \mathbb{N}}$-almost surely, the collection of points $(Y_{1,x_0}, \ldots , Y_{\lceil r^{-A}\rceil,x_0})$ is $r$-dense in $X(r^{-1})$ for $r > 0$ small enough depending on the sequence. 
	\end{theorem}
	
	An $r$-dense subset of $X(r^{-1})$ needs to contain at least $O(r^{-\dim G})$ many elements.  Therefore by the pigeonhole principle the constant $A > 0$ from \eqref{ProbEstimateEffectiveDensity} must satisfy $A \geq \dim G$. We furthermore point out that Theorem~\ref{EffectiveDensity} is a consequence of Theorem~\ref{EffectiveDiameter}.
	
	It is a well-known open problem to prove error rates for \eqref{ConvolutionEquidistribution} and \eqref{OrbitEquidistribution}. While our methods only suffice to establish the above discussed effective density theorems, we show, as stated in Theorem \ref{MainTheorem}, effective equidistribution on $X$ with exponential error terms for sequences of measures of the form $\mu^{*n} * \mu_D^{*\beta\cdot n} * \delta_{x_0}$, where $\mu_D$ is a measure that satisfies sufficiently strong Diophantine properties and is supported close enough to the identity in terms of $\mu$. This is achieved by using the $L^2$-flattening results of \cite{BoutonnetIoanaSalehiGolsefidy2017} together with a novel argument passing from high dimension to effective equidistribution in the setting of random walks on homogeneous spaces. 
	
	The difficulty in the latter argument is that a Fourier inversion formula on $X$ is either not available or involves complicated terms and, in contrast to compact groups, the spectral gap of $\mu$ appears not to be useful in controlling the arising error terms. On the other hand, our method is related to the ideas by Venkatesh \cite[Section 3.1]{Venkatesh2010}, which were recently used by Mohammadi-Lindenstrauss \cite{LindenstraussMohammadi2022} to pass from high dimension to effective equidistribution for unipotent actions. In our setting, for a given measure $\mu$ on $G$ and a sequence of measures $\nu_n$ on $X$,  we show (Corollary~\ref{HighDimensiontoQuantitativeEquidistribution}), using the spectral gap of $\mu$, that the sequence $\mu^{*n} * \nu_n$ equidistributes effectively under the assumption that $\nu_n$ has dimension (see \eqref{HighDimensionProperty}) at least $\dim G - \gamma$ for $\gamma = \gamma(\mu) > 0$ a constant depending on $\mu$. This method is the main contribution of this paper and allows us to deduce Theorem~\ref{MainTheorem}, upon which Theorem 1.1 relies. 
	
	The current techniques prevent us from proving effective equidistribution of $\mu^{*n}_D * \delta_{x_0}$ for a measure $\mu_D$ satisfying sufficiently strong Diophantine properties, as we are presently not able to control $\gamma(\mu_D)$.
	
	In \cite{Kogler2022}, the second author introduced the notion of a $(c_1,c_2,\eps)$-Diophantine measure in order to conveniently capture the flattening results of \cite{BoutonnetIoanaSalehiGolsefidy2017}. We utilize the same definition in this paper and refer to \cite{Kogler2022} for a discussion. 
	
	\begin{definition}
		Let $G$ be a connected Lie group, $\mu_D$ a probability measure on $G$ and let $c_1, c_2, \varepsilon > 0$. A measure $\mu_D$ is called \textbf{$(c_1, c_2, \varepsilon)$-Diophantine} if 
		\begin{enumerate}[(i)]
			\item $\mu_D$ is $(c_1 \log \frac{1}{\varepsilon}, c_2 \log\frac{1}{\varepsilon})$-Diophantine, i.e. for $n$ large enough, $$\sup_{H < G} \mu_D^{*n}(B_{\varepsilon^{c_1 n}}(H)) \leq \varepsilon^{c_2 n},$$ where the supremum is taken over all connected closed subgroups. 
			\item $\mathrm{supp}(\mu_D) \subset B_{\varepsilon}(e)$.
		\end{enumerate}
	\end{definition}
	
	The employed flattening results for $(c_1,c_2,\varepsilon)$-Diophantine measures (Corollary 4.2 of \cite{BoutonnetIoanaSalehiGolsefidy2017} and stated in Proposition~\ref{SuperFlatteningLemma} of Section \ref{SectionEffectiveEquidistribution}) were initiated by Bourgain in his construction of a monotone expander \cite{Bourgain2009} (see also \cite{BourgainYehudayoff2013}) and of a finitely supported measure on $\mathrm{SL}_2(\R)$ with absolutely continuous Furstenberg measure \cite{Bourgain2012}. In \cite{BoutonnetIoanaSalehiGolsefidy2017} the flattening results are used to establish a local spectral gap and, among other applications, to generalize Bourgain's examples of monotone expanders. The second author \cite{Kogler2022} used these results to prove a local limit theorem for $G$ acting on its associated symmetric space for $(c_1,c_2,\varepsilon)$-Diophantine measures, where $\varepsilon$ is sufficiently small in terms of $c_1$ and $c_2$. We further mention that by \cite{Lequen2022} and independently \cite{Kogler2022}, these measures have absolutely continuous Furstenberg measure, generalizing Bourgain’s example \cite{Bourgain2012} to simple Lie groups. 
	
	We state a result from \cite{BoutonnetIoanaSalehiGolsefidy2017} showing that there is an abundant collection of examples of finitely supported $(c_1, c_2, \varepsilon)$-Diophantine measures for arbitrarily small $\varepsilon$.
	
	\begin{theorem}(Theorem 3.1 of \cite{BoutonnetIoanaSalehiGolsefidy2017})\label{BIG17Theorem3.1}
		Let $G$ be a connected simple Lie group with finite center. Let $\Gamma < G$ be a countable dense subgroup and assume $\mathrm{Ad}(\Gamma)$ consists of matrices with algebraic entries with respect to a basis of $\mathfrak{g}$.
		
		Then there exist $c_1, c_2 > 0$ such that for every $\varepsilon_0 > 0$ there is $0 < \varepsilon <  \varepsilon_0$ and a finitely supported symmetric $(c_1, c_2, \varepsilon)$-Diophantine probability measure $\mu_D$ satisfying $\supp{\mu_D} \subset \Gamma \cap B_{\varepsilon}(e)$.
	\end{theorem}

	To state our effective equidistribution result, denote for a continuous function $f$ on $X$ by $\mathrm{Lip}(f)$ the Lipschitz constant of $f$ defined in \eqref{LipschitzConstant} and by $\mathrm{Lip}(X)$ the space of Lipschitz functions on $X$. The $L^{\infty}$-norm on $X$ is written as $||\cdot||_{\infty}$.
	
	For a bounded operator $T$, the quantity $\rho(T)$ is the spectral radius of $T$. Let $\pi_X$ be the Koopman representation of $X$. We say that a probability measure $\mu$ on $G$ has a spectral gap on $X$ if $$\rho(\pi_X(\mu)|_{L_0^2(X)}) = \lim_{n \to \infty} \sqrt[n]{||\pi_X(\mu)|^n_{L_0^2(X)}||} < 1,$$ i.e. where $||\cdot||$ is the operator norm. It follows from Theorem C of \cite{Shalom2000} that if $G$ is non-compact, then a measure $\mu$ on $G$ that is not supported on a closed amenable subgroup has a spectral gap on $X$ (see Lemma~\ref{muspectralgap}). 
	
	\begin{theorem}\label{MainTheorem}
		Let $G$, $\Lambda$ and $X$ be as in Theorem~\ref{EffectiveDiameter}. Let $\mu$ be a compactly supported probability measure on $G$ with a spectral gap on $X$ and let $c_1, c_2 > 0$. 
		
		Then there are $\varepsilon_0 = \varepsilon_0(\mu,c_1,c_2) > 0$ and $\theta=\theta(\mu)$ such that for every $(c_1, c_2 , \varepsilon)$-Diophantine probability measure $\mu_D$ with $0<\varepsilon \leq \varepsilon_0$ the following holds: There exists $\beta = \beta(\mu,\eps)$ such that for every bounded Lipschitz function $f \in \mathrm{Lip}(X)$, $x_0 \in X$ and $n\geq 1$, $$\int f(g x_0) \, d(\mu^{*n}*\mu_D^{*\beta \cdot n})(g) = \int f \, d\Haarof{X} + O_{\Lambda, \mu, \mu_D}((\mathrm{Lip}(f) + \operatorname{ht}(x_0)||f||_{\infty}) e^{- \theta n}).$$
	\end{theorem}
	
	We emphasize that in the statement of Theorem~\ref{MainTheorem} $\mu$ is an arbitrary measure with a spectral gap on $X$ and that $\eps_0$ depends on $\mu$. In the outline of proofs in Section~\ref{SectionOutline}, it is exposed why this is necessary. Moreover, for convenience we use the convention to make no notational distinction between the possibly non-integer number $\beta \cdot n$ and the closest integer to it. In addition, we mention that we can give an explicit lower bound (see Theorem~\ref{HighDimensiontoQuantitativeEquidistribution0}) for the decay rate $\theta = \theta(\mu)$, depending on the spectral gap of $\mu$ and on the size of $\mathrm{supp}(\mu)$.
	
	All the previous results build on Theorem~\ref{MainTheorem} as Theorem~\ref{EffectiveDiameter} follows from Theorem~\ref{BIG17Theorem3.1} and Theorem~\ref{MainTheorem} if $G$ is non-compact.
	
	\subsection*{Structure of paper} We give an outline of proofs and summarize the notation and constants used in Section~\ref{SectionOutline}. In Section~\ref{SectionQuantitativeNonDivergence} we discuss the height function and quantitative non-divergence. Section~\ref{SectionEffectiveEquidistribution} is devoted to Theorem~\ref{MainTheorem} and in Section~\ref{SectionEffectiveDensity} we prove Theorem~\ref{EffectiveDiameter} and Theorem~\ref{EffectiveDensity}.
	
	\subsection*{Acknowledgement} The authors are grateful to Péter Varjú for encouraging us to work on the discussed circle of problems. Moreover, we thank Timothée Bénard, Emmanuel Breuillard, Amitay Kamber, Greg Kuperberg, Félix Lequen, Péter Varjú, Weikun He and the anonymous referee for comments on a preliminary version.

	\section{Outline and Notation}\label{SectionOutline}
	
	\subsection{Outline of proofs}
	
	As Theorem~\ref{EffectiveDiameter} relies on Theorem~\ref{MainTheorem}, we describe the proof of Theorem~\ref{MainTheorem} first. It comprises two steps and for simplicity of this outline we assume that $X$ is compact. First, we use the $L^2$-flattening results by \cite{BoutonnetIoanaSalehiGolsefidy2017} to show that after a small number of steps, the measure $\nu_n = \mu_D^{*n} * \delta_{x_0}$ has high dimension. Indeed, we will show in Proposition~\ref{HighDimension} that for any $c_1,c_2,\gamma > 0$ and every $(c_1,c_2, \varepsilon)$-Diophantine probability measure $\mu_D$ for $\varepsilon$ sufficiently small in terms of $c_1,c_2$ and $\gamma$ it holds,
	\begin{equation}\label{HighDimensionProperty}
		\nu_n(B_\delta^X(x))  \leq \delta^{\dim G - \gamma}
	\end{equation}
	for any $\delta > 0$ small enough, $x \in X$ and $n \asymp C_0'\frac{\log \frac{1}{\delta}}{\log \frac{1}{\varepsilon}} $, where $C_0' = C_0'(c_1,c_2,\gamma)$. We refer to \eqref{HighDimensionProperty} as having high dimension since $\Haarof{X}(B_{\delta}(x)) \asymp \delta^{\dim G}$ and therefore $\nu_n$ behaves comparable to $\Haarof{X}$ after only $O(\log \frac{1}{\delta})$ many steps. The $L^2$-flattening results of \cite{BoutonnetIoanaSalehiGolsefidy2017} imply \eqref{HighDimensionProperty} on $G$ (see \eqref{LinftySuperFlattening}) and we are able to pass from high dimension on $G$ to high dimension on $X$ using that $\mu_D$ is supported close to the identity.
	
	The second step is the main new contribution of this paper and amounts to deducing quantitative equidistribution from high dimension \eqref{HighDimensionProperty}. This step can be seen as an analogue of the Sarnak-Xue trick \cite{SarnakXue1991}, as used by \cite{BourgainGamburd2008Invent}, \cite{BourgainGamburd2008Annals}, \cite{deSaxce2013} \cite{BenoistDeSaxce2016}, or of the methods by Venkatesh \cite[Section 3.1]{Venkatesh2010}, as applied by \cite{LindenstraussMohammadi2022} and \cite{LindenstraussMohammadiWang2022}. In our setting, we exploit the spectral gap of $\mu$ on $X$. Write 
	\begin{equation}\label{gapdef}
		\mathrm{gap}(\mu) = -\log \rho(\pi_X(\mu)|_{L^2_0(X)})
	\end{equation}
	such that for $f_1, f_2 \in L^2(X)$ and any $0 < c < \mathrm{gap}(\mu)$,
	\begin{equation}\label{ExponentialMixing}
		|\langle \pi_X(\mu)^n f_1, f_2 \rangle - \langle f_1,1 \rangle \langle 1, f_2 \rangle | \leq e^{-c \cdot n}||f_1||_2 ||f_2||_2
	\end{equation} for sufficiently large $n$.   
	
	The proof of Theorem \ref{MainTheorem} proceeds by reducing to functions with $\int f \, dm_X = 0$. We use the spectral gap of $\mu$ to upgrade high dimension of $\nu_n =  \mu_D^{*\beta \cdot n} * \delta_{x_0}$ to effective equidistribution of $\mu^{*n} * \nu_n$. We introduce the additional parameter $\beta$ in order to ensure that $\beta \cdot n \asymp C_0'\frac{\log \frac{1}{\delta}}{\log \frac{1}{\varepsilon}}$. We then write $$\int f(gx_0) \, d(\mu^{*n}*\mu_D^{*\beta \cdot n})(g) = \int \pi_X^n(\mu)f \, d\nu_{n}.$$ In order to apply the spectral gap, one approximates the latter integral with the inner product $$\langle \pi_{X}(\mu)^n f, h_{n,\delta} \rangle = \int \pi_X(\mu)^n f  \cdot h_{n,\delta} \, d\Haarof{X} \quad\text{ for }\quad h_{n, \delta}(x) = \frac{\nu_{n}(B_{\delta}(x))}{\Haarof{G}(B_{\delta})}.$$ In Section~\ref{SectionCompletionMainTheorem} it is shown that this approximation is possible up to an error of size $$O(\delta \cdot \mathrm{Lip}(\pi_X(\mu)^n f)) = O(\delta \cdot e^{O(R(\mu))n} \cdot \mathrm{Lip}(f))$$ for $R(\mu) = \min\{R>0: \operatorname{supp}(\mu) \subseteq B_R\}.$ We then use the spectral gap to bound $|\langle \pi_{X}(\mu)^n f, h_{n,\delta} \rangle| \leq e^{-cn} \cdot ||f||_2 \cdot ||h_{n,\delta}||_2$ and finally high dimension \eqref{HighDimensionProperty} to show that $||h_{n,\delta}||_2 \ll \delta^{-2\gamma}$. Choosing $\delta$ decaying appropriately in $n$, Theorem \ref{MainTheorem} follows. A further difficulty is to show Theorem \ref{MainTheorem} also for non-compact $X$, which we deal with results from section \ref{SectionQuantitativeNonDivergence}.
	
	A careful analysis of the above error terms reveals that we need to set $\gamma \ll_{\Lambda} \min(\frac{\mathrm{gap}(\mu)}{R(\mu)},1)$. The latter choice determines how small $\eps$ needs to be in order for \eqref{HighDimensionProperty} to hold. On the other hand, in \cite{BourgainGamburd2008Invent}, \cite{BourgainGamburd2008Annals} and \cite{BenoistDeSaxce2016}, in order to deduce spectral gap from high dimension, it is necessary to be at dimension $\dim G - \gamma$ for $\gamma > 0$ an absolute constant depending only on $G$. In our case the measure $\mu$ determines $\gamma > 0$. This difference prevents us currently from proving effective equidistribution of $\mu_D^{*n} * \delta_{x_0}$, for $\eps$ small enough in $c_1,c_2$ as this requires us to control the size of $\mathrm{gap}(\mu_D)$.  
	
	All further results in this paper rely on Theorem~\ref{MainTheorem}. Indeed, Theorem~\ref{EffectiveDiameter} follows from Theorem~\ref{BIG17Theorem3.1} and Theorem~\ref{MainTheorem}. To explain the idea of the proof of Theorem~\ref{EffectiveDensity}, we fix $y \in X$ and $r > 0$ and discuss how to show that at least one of the points $Y_{1,x_0}, \ldots , Y_{\lceil r^{-A} \rceil, x_0}$ is contained in $B_{r}(x)$ with high probability. Note that by Theorem~\ref{EffectiveDiameter} for any $x\in X$ and $r > 0$ there is $g \in S^{O(\log r^{-1})}$ such that $gx \in B_{r}(y)$. This may be used to show that for any $x \in X$, 
	\begin{equation}\label{HittingProbAfterLogSteps}
		\mathbb{P}[Z_{\lceil C \log r^{-1} \rceil} \cdots Z_1x \in B_{r}(y)] \geq r^B  
	\end{equation}
	for constants $B,C > 0$. To complete the proof, we consider the sequence $$Y_{1,x_0}, \ldots , Y_{\lceil r^{-A} \rceil, x_0}.$$ We may split the rather long interval $1, \ldots , \lceil r^{-A} \rceil$ into short intervals of length $C\log r^{-1}$ resulting in $\frac{\lceil r^{-A} \rceil}{C\log r^{-1}} \geq r^{-3A/4}$ many such intervals for $r$ small enough. Applying \eqref{HittingProbAfterLogSteps} to each of these intervals, we can bound the probability that none of the points  $Y_{1,x_0}, \ldots , Y_{\lceil r^{-A} \rceil, x_0}$ are contained in $B_{r}(x)$ by $(1 - r^B)^{r^{-3A/4}} \leq \exp(-r^{-A/2})$ for $r$ small enough and $A$ large enough. In the compact case, this argument applies (see Theorem~\ref{CompactEffectiveDensity}). However, if $X$ is non-compact, as we additionally are required to consider non-divergence, we only arrive at the weaker bound \eqref{ProbEstimateEffectiveDensity}.
	
	\subsection{Notation}\label{Notation} Throughout this paper, $G$ denotes a connected simple Lie group with finite center. We use the asymptotic notation $X \ll Y$ or $X = O(Y)$ to denote that $|X| \leq CY$ for a constant $C > 0$. If the constant $C$ depends on additional parameters we add subscripts, unless the quantity depends on the fixed group $G$ in which case we omit additional subscripts for convenience.  Moreover, $X \asymp Y$ denotes $X \ll Y$ and $Y \gg X$.  
	
	Denote by $d: G \times G \to \R_{\geq 0}$ a right invariant metric. For $R > 0$ and $h \in G$, write $$B_R(h) = \{ g \in G \,:\, d(g,h) < R \}.$$ For $x = e$, we abbreviate $B_R = B_R(e)$. Moreover, for any closed subset $H \subset G$, we define $$B_R(H) = \{ g \in G \,:\, d(g,H) < R \},$$ where $d(g,H) = \inf_{h \in H} d(g,h)$. 
	
	Let $\Haarof{G}$ be a Haar measure on $G$. Then as $B_{R}(x) = B_{R} \cdot x$, it holds that $\Haarof{G}(B_{R}(x)) = \Haarof{G}(B_{R})$.
	
	We fix a basis of $\mathfrak{g}=\operatorname{Lie}(G)$, inducing the associated euclidean norm $||\cdot||$ on $\mathfrak{g}$. For $g\in G$ we define
	$$\|g\|:=\max_{i,j}\{|\operatorname{Ad}(g)_{ij}|,|\operatorname{Ad}(g^{-1})_{ij}|\},$$
	where $\operatorname{Ad}(g)_{ij}$ and $\operatorname{Ad}(g^{-1})_{ij}$ denote the matrix coefficients of $\operatorname{Ad}(g)$ and $\operatorname{Ad}(g^{-1})$ respectively with respect to the basis we fixed. Note that
	\begin{equation}\label{NormProperties}
		\|g^{-1}\|=\|g\|,\quad\|g_1g_2\|\ll\|g_1\|\|g_2\|,\quad\|\operatorname{Ad}(g)\|_{\operatorname{op}}\ll \|g\| 	\end{equation}
	for any $g,g_1,g_2\in G$, where $\|\operatorname{Ad}(g)\|_{\operatorname{op}}$ denotes the operator norm of the adjoint action of $g$ with respect to the chosen euclidean norm on $\mathfrak{g}$. 
	
	For a lattice $\Lambda < G$, write $X = G/\Lambda$ with the endowed metric $d_X(x,y) = \inf_{\lambda \in \Lambda} d(g_x\lambda,g_y)$ for $g_x, g_y \in G$ such that $x = g_x \Lambda$ and $y = g_y \Lambda$. For $R > 0$ and $x \in X$ we denote $$B_R^X(x) = \{ y \in X \,:\, d_X(x,y) < R \}. $$ For $x\in X$ we denote by $\operatorname{inj}(x)$ the maximal injectivity radius at $x\in X$, which is the supremum of $r>0$ such that the map $g\mapsto gx$ is an isometry from $B_r(\operatorname{id})$ to $B_r(x)$.
	
	We say that a function $f \in C(X)$ is Lipschitz if its Lipschitz constant 
	\begin{equation}\label{LipschitzConstant}
		\mathrm{Lip}(f) =  \sup_{x,y \in X, x \neq y} \frac{|f(x) - f(y)|}{d_X(x,y)}
	\end{equation}
	is finite. We denote by $\mathrm{Lip}(X) \subset C(X)$ the space of Lipschitz functions on $X$. 
	
	\subsection{Constants}
	
	For convenience we list here the constants used in this paper. We denote by $E_1, E_2, \ldots $ constants depending only on $G$ and $\Lambda$:
	\begin{enumerate}
		\item $E_1$ is defined in Proposition~\ref{heightftn} (1).
		\item $E_2$ is defined in Proposition~\ref{heightftn} (3).
		\item $E_3$ is only used in the proof of Proposition~\ref{heightftn}.
		\item $E_4$ is defined in Proposition~\ref{HighDimension}.
		\item $E_5$ is the constant depending only $G$ such that $\Haarof{G}(B_r) \ll e^{E_5r}$ for all $r > 0$.
		\item $E_6$ is the constant depending only $G$ such that $||g|| \ll e^{E_6R}$ for all $g\in B_R$.
		\item $E_7$ is defined in Theorem~\ref{HighDimensiontoQuantitativeEquidistribution}.
		\item $E_8:= E_2^{-1}E_1$ is introduced in Lemma~\ref{MeasureApproximation}.
	\end{enumerate}
	Denote by $\kappa_1, \kappa_2, \ldots $ further constants depending only on $G$ and $\Lambda$:
	\begin{enumerate}
		\item $\kappa_1$ is defined in Proposition~\ref{heightftn} (1).
		\item $\kappa_2$ is defined in Proposition~\ref{heightftn} (2).
		\item $\kappa_3$ is only used in the proof of Proposition~\ref{heightftn}.
	\end{enumerate}
	
	We often don't introduce constants for quantities that depend on $\mu$ or further quantities, yet we do so in the following cases:
	\begin{enumerate}
		\item $\varepsilon_0 = \varepsilon_0(c_1,c_2,\gamma)$ is from Theorem~\ref{MainTheorem}.
		\item $C_0 = C_0(c_1,c_2,\gamma)$ is from the proof of Theorem~\ref{HighDimensiontoQuantitativeEquidistribution0}.
		\item $\varepsilon_0' = \varepsilon_0(c_1,c_2,\gamma)$ is from Proposition~\ref{HighDimension}.
		\item $C_0' = C_0'(c_1,c_2,\gamma)$ is from Proposition~\ref{HighDimension}.
		\item $\varepsilon_0'' = \varepsilon_0(c_1,c_2,\gamma)$ is from Proposition~\ref{SuperFlatteningLemma}.
		\item $C_0'' = C_0'(c_1,c_2,\gamma)$ is from is from Proposition~\ref{SuperFlatteningLemma}.
		\item $\mathrm{gap}(\mu) = -\log ||\pi_X(\mu)|_{L^2_0(X)}||$.
		\item $R(\mu) = \min\{R>0: \operatorname{supp}(\mu) \subseteq B_R\}.$
	\end{enumerate}

	\section{Proof of Theorem~\ref{MainTheorem}}\label{SectionEffectiveEquidistribution}
	
	The reader may recall the outline of the proof of Theorem~\ref{MainTheorem} given in Section~\ref{SectionOutline}. We first discuss quantitative non-divergence in Section~\ref{SectionQuantitativeNonDivergence}. In Section~\ref{SectionHighDimension}, we show that $\nu_n = \mu_D^{*n} * \delta_{x_0}$ has high dimension on $X$ for a suitable constant $\beta$. This exploits flattening results on $G$ and that $\mu_D$ is supported close to the identity. Then, in Section~\ref{SectionCompletionMainTheorem}, we use the spectral gap of $\mu$ to deduce effective equidistribution of $\mu^{*n}*\nu_n$, assuming that $\mu_D$ has sufficiently strong Diophantine properties and is close enough to the identity in terms of $\mu$. 
	
	\subsection{Quantitative Non-Divergence}\label{SectionQuantitativeNonDivergence}
	We discuss the height function $\mathrm{ht} : X \to \R_{\geq 1}$ as mentioned in the introduction. If $X$ is compact, we set $\mathrm{ht} \equiv 1$ and for the remainder of this section we assume that $X$ is non-compact. For the purposes of the later sections, we need to quantitatively control the recurrence to compact subsets. The following proposition follows from constructions of $\mathrm{ht}$ by \cite{EskinMargulis2004} and by \cite{BenoistQuint2011} and \cite{BenoistQuint2012}. 
	
	\begin{proposition}\label{heightftn}
		Let $G$ be a semisimple Lie group, $\Lambda < G$ a non-uniform irreducible lattice in $G$ and $X = G/\Lambda$. There exists a proper continuous function $\operatorname{ht}:X\to \R_{\geq 1}$ such that the following properties hold. All the constants in this proposition depend on $G$ and $\Lambda$.
		\begin{enumerate}
			\item There exist $0<\kappa_1<1$ and $E_1>0$ such that $\operatorname{ht}(x) \geq E_1\operatorname{inj}(x)^{-\kappa_1}$ for any $x\in X$.
			\item There exists $\kappa_2>0$ such that for all $x\in X$ there is $g \in G$ such that $x = g\Lambda$ and $\|g\|\ll \operatorname{ht}(x)^{\kappa_2}$.
			\item (Log-Lipschitz condition) There exists $E_2>1$ such that for any $g\in B_1^G(e)$ and $x \in X$,
			\begin{equation}\label{LogLipschitz}
				E_2^{-1}\operatorname{ht}(x)\leq\operatorname{ht}(gx)\leq E_2\operatorname{ht}(x).
			\end{equation}
			\item (Contraction Hypothesis) Suppose that $\mu$ is a compactly supported probability measure on $G$ whose support generates a Zariski-dense semigroup. Then there exists $0<a=a(\Lambda,\mu)<1$, $b=b(\Lambda,\mu)>0$, and $N=N(\Lambda,\mu)\in\mathbb{N}$ such that for any $x\in X$
			\begin{equation}\label{ContractionHypothesis}
				\int \operatorname{ht}(gx) \, d\mu^{*N}(g) \leq a\operatorname{ht}(x) + b.
			\end{equation}
		\end{enumerate}
	\end{proposition}
	
	In particular, Proposition \ref{heightftn} gives the following quantitative non-divergence result, which is due to \cite[Lemma 3.1]{EskinMargulis2004}, yet we reconstruct the proof to make the dependence on the height explicit.
	
	\begin{lemma}\label{QnD}
		Suppose that $\mu$ is a compactly supported probability measure on $G$ whose support generates a Zariski-dense semigroup. Then for any $n\ge 1$, $x\in X$, and $\mathbf{h}>0$,
		\begin{equation}\label{QnDestimate}
			\mu^{*n}(\{g\in G: \operatorname{ht}(gx)\ge \mathbf{h}\})\ll_{\Lambda, \mu} \mathbf{h}^{-1} \cdot  \operatorname{ht}(x).
		\end{equation}
	\end{lemma}
	
	\begin{proof}
		Let $0<a<1$, $b>0$, and $N\in\mathbb{N}$ be as in (4) of Proposition \ref{heightftn}. For any $n\ge 1$, we may write $n=qN+r$ for some $q\in\mathbb{Z}_{\ge 0}$ and $0\leq r\leq N-1$. By (3) of Proposition \ref{heightftn}, there exists $C_1=C_1(\Lambda,\mu)>1$ such that
		$$C_1^{-1}\operatorname{ht}(x)\leq\operatorname{ht}(gx)\leq C_1\operatorname{ht}(x)$$
		for any $g\in \bigcup_{\ell = 1}^N\operatorname{supp}(\mu^{*\ell})$. Thus
		\begin{equation}\label{qNbound}
			\mu^{*n}(\{g\in G: \operatorname{ht}(gx)\ge \mathbf{h}\})\leq\mu^{*(qN)}(\{g\in G: \operatorname{ht}(gx)\ge C_1^{-1}\mathbf{h}\}).
		\end{equation}
		On the other hand, after iterating \eqref{ContractionHypothesis} and summing the geometric series, we get
		\begin{equation}\label{iteration}
			\int \operatorname{ht}(gx) \, d\mu^{*qN}(g) \leq a^{q}\operatorname{ht}(x)+(b+ab+\cdots+a^{q-1}b)\leq \operatorname{ht}(x)+\frac{b}{1-a}
		\end{equation}
		for any $q\in\mathbb{N}$. It follows that
		\begin{equation}\label{iteration}
			\begin{aligned}
				\mu^{*(qN)}(\{g\in G: \operatorname{ht}(gx)\ge C_1^{-1}\mathbf{h}\})&\leq (C_1^{-1}\mathbf{h})^{-1}\int \operatorname{ht}(gx) \, d\mu^{*qN}(g)\\
				&\leq C_1\mathbf{h}^{-1}(\operatorname{ht}(x)+C_2),
			\end{aligned}
		\end{equation}
		where $C_2=\frac{b}{1-a}$. Combining \eqref{qNbound} and \eqref{iteration}, we obtain \eqref{QnDestimate} as $\mathrm{ht} \geq 1$.
	\end{proof}

	\subsection{High Dimension}\label{SectionHighDimension}
	
	In this section we establish that $(c_1, c_2,\varepsilon)$-Diophantine measures have high dimension on $X$. We note that for a fixed $(c_1, c_2,\varepsilon)$-Diophantine measure, we can only get to dimension close to $\dim G$, yet not arbitrarily close to $\dim G$.
	
	\begin{proposition}(High Dimension)\label{HighDimension}
		Let $G$ be a connected simple Lie group with finite center and let $\Lambda < G$ be a lattice. Let $\gamma, c_1, c_2 > 0$. Then there exist $\varepsilon_0' = \varepsilon_0'(c_1, c_2, \gamma)$ and $C_0'=C_0'(c_1,c_2,\gamma)$ such that every $(c_1, c_2, \varepsilon)$-Diophantine probability measure $\mu_D$ with $0<\varepsilon\leq\varepsilon_0'$ satisfies the following. 
		
		For $\nu_n = \mu_D^{*n}*\delta_{x_0}$ with $x_0 \in X$ and $\delta > 0$ small enough,  $$\nu_n(B_{\delta}^X(x)) \leq \mathrm{ht}(x)^{E_4}\delta^{\dim G - \gamma}$$ for $n \asymp C_0'\frac{\log \frac{1}{\delta}}{\log \frac{1}{\varepsilon}}$, where $E_4>0$ is a constant depending on $G$.
	\end{proposition}
	
	Proposition~\ref{HighDimension} uses the strong flattening results of \cite{BoutonnetIoanaSalehiGolsefidy2017}. To introduce notation, denote $$P_{\delta} = \frac{1_{B_{\delta}}}{\Haarof{G}(B_{\delta})}$$ and observe that for any symmetric measure $\nu$ it holds that $(\nu*P_{\delta})(g) = \frac{\nu(B_{\delta}(g^{-1}))}{\Haarof{G}(B_{\delta})}$ for any $g \in G$. 
	
	\begin{proposition}(Flattening Lemma, follows from Corollary 4.2 of \cite{BoutonnetIoanaSalehiGolsefidy2017})\label{SuperFlatteningLemma}
		Let $c_1, c_2 > 0$. Then for every $\gamma > 0$ there is $\varepsilon''_0 = \varepsilon''_0(c_1,c_2, \gamma)  > 0$ and $C_0'' = C_0''(c_1,c_2, \gamma) > 0$ such that the following holds.
		
		If $\varepsilon \leq \varepsilon''_0$ and $\mu_D$ is a symmetric and $(c_1,c_2, \varepsilon)$-Diophantine probability measure on $G$, then for $\delta > 0$ small enough, 
		\begin{equation}\label{LinftySuperFlattening}
			||\mu_D^{*n} * P_{\delta}||_{\infty} \leq \delta^{- \gamma} \quad\quad \text{ for any integer } \quad\quad n \geq C_0'' \frac{\log \frac{1}{\delta}}{\log \frac{1}{\varepsilon}}.
		\end{equation}
	\end{proposition}
	
	\begin{proof}
		By Corollary 4.2 of \cite{BoutonnetIoanaSalehiGolsefidy2017}, there exist $\varepsilon_1 = \varepsilon_1(c_1,c_2, \gamma)  > 0$ and $D_0 = D_0(c_1,c_2, \gamma) > 0$ such that 
		\begin{equation}\label{L2SuperFlattening}
			||\mu_D^{*n} * P_{\delta}||_{2} \leq \delta^{- \gamma} \quad\quad \text{ for any integer } \quad\quad n \geq D_0 \frac{\log \frac{1}{\delta}}{\log \frac{1}{\varepsilon}}.
		\end{equation} for $\delta > 0$ small enough provided that $0 < \varepsilon \leq \varepsilon_1$ and $\mu_D$ is a symmetric and $(c_1,c_2, \varepsilon)$-Diophantine probability measure on $G$. 
		
		We now deduce the claimed $L^{\infty}$-estimate for $\mu_D^{*n} * P_{\delta}$. To establish the latter we notice that as in Lemma 2.5 of \cite{BourgainGamburd2008Invent} one shows that $$ P_{\delta} \ll P_{\delta} * P_{\delta} \ll P_{2\delta},$$ for absolute implied constants depending on $G$. Notice that as $P_{\delta}(g) = P_{\delta}(g^{-1})$ for all $g \in G$ and, as $\mu_D$ is symmetric, it follows that $(\mu_D^{*n}*P_{\delta})(g) = (P_{\delta} * \mu_D^{*n})(g^{-1})$.
		Therefore, $$||\mu_D^{*2n}*P_{\delta}||_{\infty} \ll ||\mu_D^{*2n} * P_{\delta} *P_{\delta}||_{\infty} = ||\mu_D^{*n} * P_{\delta} * \mu_D^{*n} * P_{\delta}||_{\infty} \leq ||\mu_D^{*n}*P_{\delta}||_2^2 \leq \delta^{-2\gamma},$$ having applied Cauchy-Schwartz in the penultimate inequality. Replacing $\gamma$ by $\gamma/4$, and choosing $\eps_0''$ and $C_0''$ suitably in terms of $\eps_1$ and $D_0$, the claim follows using that for any two probability measures $\nu_1, \nu_2$ on $G$ it holds that $||\nu_1 * \nu_2 * P_{\delta}||_{\infty} \leq ||\nu_2 *P_{\delta}||_{\infty}$. 
	\end{proof}

	Notice that \eqref{LinftySuperFlattening} implies that $\mu_D^{*n}(B_{\delta}(g)) \ll  \delta^{\dim G - \gamma}$ for any $g \in G$ and $\delta$ sufficiently small. The content of Proposition~\ref{HighDimension} is therefore to show the same conclusion on $X$. This is achieved by exploiting that $\mu_D$ is supported close to the identity. 
	
	We proceed with a few preliminary lemmas.
	
	\begin{lemma}\label{SizeLeftTranslate}
		Let $g \in G$ and $\delta > 0$ be small enough (in terms of $G$). Then for $y \in G$ with $||y|| \ll \delta^{-1}$, $$yB_{\delta}(g) \subset B_{2||y||\delta}(yg), \qquad B_{\delta}(yg) \subset yB_{2||y||\delta}(g).$$
	\end{lemma}
	
	\begin{proof}
		As $B_R(g) = B_R \cdot g$, it suffices to prove the claims for $g = e$. For the first claim, choose $c > 0$ sufficiently small such that for every element $g \in B_c \subset G$ the exponential map has a unique preimage and $\frac{1}{2}||\exp^{-1}(g)|| \leq d(g,e) \leq 2||\exp^{-1}(g)||.$ Choose $\delta < c$ and let  $h \in B_{\delta}$. Then there is a unique $X \in \mathfrak{g}$ with $||X|| \leq 2\delta$ such that $h = \exp(X)$ and moreover for any $y \in G$, it holds that $yhy^{-1} = \exp(\Ad(y)X).$ As $||\Ad(y)X|| \leq ||\Ad(y)||_{\mathrm{op}}||X|| \ll ||y|| \, ||X||$, it therefore holds for $||y|| \ll \delta^{-1}$ that $||\Ad(y)X|| < c$ and hence $d(yh,y) = d(yhy^{-1},e) \leq 2||\Ad(y)X|| \leq 2||y|| \delta$, showing the first claim. The second claim follows form the first since $y^{-1}B_{\delta}(yg) \subset B_{2 ||y||\delta}(y^{-1}y g) = B_{2 ||y||\delta}(g)$.
	\end{proof}
	
	\begin{lemma}\label{LatticeCountingUpperBound}
		Let $\Lambda < G$ be a lattice. Then for any $g,h\in G$, $$|B_R \cap g\Lambda h| \ll_{\Lambda} \|g\|^{O(1)}\Haarof{G}(B_R).$$
	\end{lemma}
	
	\begin{proof}
		As a lattice is discrete, by Lemma~\ref{SizeLeftTranslate} we may choose $c \asymp_{\Lambda} \|g\|^{-1}$ such that for every $\lambda \in \Lambda$ it holds that $B_c(g\lambda h) \cap g\Lambda h= \{ g\lambda h\}$. Therefore, $$|B_R \cap g\Lambda h| = \frac{\Haarof{G}(B_R \cap B_c(g\Lambda h))}{\Haarof{G}(B_c)} \ll_{\Lambda} \|g\|^{O(1)}\Haarof{G}(B_R).$$
	\end{proof}
	
	\begin{proposition}(High Dimension of Left and Right Translates of Lattice Neighbourhood)\label{LeftRightTranslate}
		Let $\Lambda$ be a lattice in $G$. Let $c_1, c_2, \gamma > 0$. Then there are $\varepsilon'_0 = \varepsilon'_0(c_1, c_2,\gamma)$ and $C_0' = C_0'( c_1,c_2,\gamma)>0$ such that the following holds.
		
		Let $\mu_D$ be a $(c_1, c_2, \varepsilon)$-Diophantine measure for $0<\varepsilon \leq \varepsilon'_0$. Then for $\delta$ small enough, it holds for all $x,y \in G$,
		$$\mu_D^{*n}(B_{\delta}(y\Lambda x)) \ll_{\Lambda} \|y\|^{O(1)}\delta^{\dim G - \gamma}$$ for $n \asymp C_0' \frac{\log \frac{1}{\delta}}{\log \frac{1}{\varepsilon}}$.
	\end{proposition}
	
	\begin{proof}
		To prove the claim we first show for all $x,y \in G$ that $$\mu_D^{*n}(B_{\delta}(y\Lambda x)) \leq \int 1_{B_{2\delta}(y\Lambda x)}(g) (\mu_D^{*n} * P_{\delta})(g^{-1}) \, d\Haarof{G}(g).$$ Indeed, using that $(\mu_D^{*n} * P_{\delta})(g) = \frac{\mu_D^{*n}(B_{\delta}(g^{-1}))}{\Haarof{G}(B_{\delta})}$ it suffices to show that $$\Haarof{G}(B_{\delta}) \mu_D^{*n}(B_{\delta}(y\Lambda x)) \leq \int 1_{B_{2\delta}(y\Lambda x)}(g) \mu_D^{*n}(B_{\delta}(g)) \, d\Haarof{G}(g).$$  We compute,
		\begin{align*}
			\Haarof{G}(B_{\delta}) \mu_D^{*n}(B_{\delta}(y\Lambda x)) &= \int \Haarof{G}(B_{\delta}(h)) 1_{B_{\delta}(y\Lambda x)}(h) \, d\mu_D^{*n}(h) \\
			&= \int\int 1_{B_{\delta}(h)}(g) 1_{B_{\delta}(y\Lambda x)}(h) \, d\Haarof{G}(g) d\mu_D^{*n}(h) \\
			&= \int\int 1_{B_{\delta}(g)}(h) 1_{B_{\delta}(y\Lambda x)}(h) \,  d\mu_D^{*n}(h) d\Haarof{G}(g) \\
			&\leq \int\int 1_{B_{\delta}(g)}(h) 1_{B_{2\delta}(y\Lambda x)}(g) \,  d\mu_D^{*n}(h) d\Haarof{G}(g) \\
			&= \int 1_{B_{2\delta}(y\Lambda x)}(g) \mu_D^{*n}(B_{\delta}(g)) \, d\Haarof{G}(g),
		\end{align*} using Fubini's Theorem and in the penultimate line that for fixed $g$, it holds that $1_{B_{\delta}(g)}(h) 1_{B_{\delta}(y\Lambda x)}(h)  \leq 1_{B_{\delta}(g)}(h) 1_{B_{2\delta}(y\Lambda x)}(g)$ as for $h \in B_{\delta}(g)\cap B_{\delta}(y\Lambda x)$ it holds that $g \in B_{2\delta}(y\Lambda x)$.
		
		Denote by $\varepsilon''_0(c_1,c_2, \frac{\gamma}{2})$ and $ C_0''(c_1,c_2,\frac{\gamma}{2}) > 0$ the constants from Proposition~\ref{SuperFlatteningLemma}. Since $G$ is simple, there exists a constant $E_5>0$ only depending on $G$ such that $m_G(B_r)\ll e^{E_5r}$ for all $r>0$. Let $C_0' = C_0'(c_1,c_2,\gamma) := C_0''(c_1,c_2,\frac{\gamma}{2})$ and $\varepsilon'_0(c_1,c_2,\gamma):=\min\big(\varepsilon''_0(c_1,c_2,\frac{\gamma}{2}),\frac{\gamma}{4C_0''E_5},e^{-1}\big)$. Employing Proposition~\ref{SuperFlatteningLemma} for $0<\varepsilon\leq\varepsilon'_0$, $\delta$ small enough and $n \asymp C_0' \frac{\log \frac{1}{\delta}}{\log \frac{1}{\varepsilon}}$, we estimate 
		\begin{align*}
			\mu_D^{*n}(B_{\delta}(y\Lambda x)) & \leq \int 1_{B_{2\delta}(y\Lambda x)}(g)  (\mu_D^{*n} * P_{\delta})(g^{-1}) \, d\Haarof{G}(g) \\ &\leq \delta^{-\frac{\gamma}{2}} \int_{B_{2n\varepsilon }} 1_{B_{2\delta}(y\Lambda x)} \, d\Haarof{G}(g) \\
			&= \delta^{-\frac{\gamma}{2}} \Haarof{G}(B_{2n\varepsilon} \cap B_{2\delta}(e)\cdot y\Lambda x).
		\end{align*}
		Therefore by Lemma \ref{LatticeCountingUpperBound},
		\begin{align*}
			\mu_D^{*n}(B_{\delta}(y\Lambda x)) &\leq  \delta^{-\frac{\gamma}{2}}\Haarof{G}(B_{2\delta}(e))|B_{2n\varepsilon} \cap y\Lambda x| \\
			&\ll_{\Lambda} \delta^{\dim G - \frac{\gamma}{2}} \|y\|^{O(1)} \Haarof{G}(B_{2n\varepsilon})\\
			&\ll_{\Lambda} \delta^{\dim G - \frac{\gamma}{2}} \|y\|^{O(1)}e^{2E_5n\varepsilon}.
		\end{align*} 
		For any $0<\varepsilon\leq\varepsilon'_0$, we have $\frac{2C_0'E_5 \varepsilon}{\log \frac{1}{\varepsilon}} \leq \frac{\gamma}{2}$, hence
		$$\mu_D^{*n}(B_{\delta}(y\Lambda x)) 
		\ll_{\Lambda}   \delta^{\dim G - \frac{\gamma}{2}} \|y\|^{O(1)} \delta^{-\frac{2C_0'E_5 \varepsilon}{\log \frac{1}{\varepsilon}}} 
		\ll_{\Lambda} \|y\|^{E_4} \delta^{\dim G - \gamma}$$ for $E_4 > 0$ an absolute constant depending only on $G$. 
	\end{proof}
	
	It is straightforward to deduce Proposition~\ref{HighDimension} from Proposition~\ref{LeftRightTranslate}
	
	\begin{proof}[Proof of Proposition~\ref{HighDimension}]
		By Proposition~\ref{heightftn} we may choose $g_0 \in G$ and $g_x \in G$ such that $x_0 = g_0 \Lambda$ and $x = g_x \Lambda$ with $\|g_x\|\ll\operatorname{ht}(x)^{\kappa_2}$. We notice that if $gx_0 \in B_{\delta}^X(x)$ then $gg_0 \in B_{\delta}(g_x \Lambda)$, or equivalently $g \in B_{\delta}(g_x\Lambda g_0^{-1})$. Therefore by Proposition~\ref{LeftRightTranslate}, for $n \asymp C_0'\frac{\log \frac{1}{\delta}}{\log \frac{1}{\eps}}$ and $\delta$ small enough, $$\nu_n(B_{\delta}^X(x)) \leq \mu_D^{*n}(B_{\delta}(g_x\Lambda g_0^{-1})) \ll_{\Lambda} \operatorname{ht}(x)^{O(1)}\delta^{\dim G - \gamma}.$$  Upon replacing $\gamma$ by $\gamma/2$ and choosing $\delta$ sufficiently small, we can remove the implied absolute constant in the last inequality and conclude the proof.
	\end{proof}
	
	\subsection{Proof of Theorem \ref{MainTheorem}}\label{SectionCompletionMainTheorem}
	
	In this section we show how to deduce quantitative equidistribution under the assumption that $\nu_n$ has a high dimension.
	
	Recall that we say that $\mu$ has a spectral gap on $X$ if $$\mathrm{gap}(\mu) = - \log \rho\big(\pi_X(\mu)|_{L_0^2(X)}\big) > 0$$ and notice that for $f_1, f_2 \in L^2(X)$ and $0 < c < \mathrm{gap}(\mu) $, 
	\begin{equation}\label{exponentialmixing}
		|\langle \pi_X(\mu)^n f_1, f_2 \rangle - \langle f_1,1 \rangle \langle 1, f_2 \rangle | \leq e^{-c\cdot n}\|f_1\|_2 \|f_2\|_2
	\end{equation} for sufficiently large $n$. We use in this section that $\mathrm{gap}(\mu) > 0$ for a large class of measures.
	
	\begin{lemma}(follows from Theorem C of \cite{Shalom2000})\label{muspectralgap}
		Let $G$ be a non-compact connected simple Lie group with finite center, let $\Lambda < G$ be a lattice, and denote by $\pi_X$ the Koopman representation on $X = G/\Lambda$. Let $\mu$ be a probability measure that is not supported on a closed amenable subgroup. Then $\mu$ has a spectral gap on $X$. 
	\end{lemma}
	
	\begin{proof}
		To apply Theorem C of \cite{Shalom2000} it suffices to show that the trivial representation $1_G$ is not weakly contained in $\pi_X|_{L_0^2(X)}$. Recall that there is $m \geq 1$ such that $\pi_X|_{L_0^2(X)}^{\otimes m}$ is weakly contained in the left regular representation $\lambda_G$. If $1_G \prec \pi_X|_{L_0^2(X)}$, it would therefore follow that $1_G = 1_G^{\otimes m} \prec \lambda_G$, which is a contradiction since $G$ is non-ameanable. 
	\end{proof}
	
	For a compactly supported probability measure $\mu$ on $G$, write
	$$R(\mu):=\min\{R>0: \operatorname{supp}(\mu)\subseteq B_R\}.$$ Denote by $E_6 > 0$ the constant depending only $G$ such that $\|g\| \ll e^{E_6R}$ for all $g\in B_R$. As $\mathrm{supp}(\mu^{*n}) \subset B_{R(\mu)n}$ for $n \geq 1$, it follows that $\|g\|\leq e^{E_6R(\mu)n}$ for any $n\ge 1$ and $g\in \operatorname{supp}(\mu^{*n})$. 
	
	We are now in a suitable position to prove Theorem \ref{MainTheorem}. For convenience we restate Theorem \ref{MainTheorem} with the addition of the below bound on $\theta$.
	
	\begin{theorem}(Theorem \ref{MainTheorem})
		\label{HighDimensiontoQuantitativeEquidistribution0}
		Let $G$, $\Lambda$ and $X$ be as in Theorem~\ref{EffectiveDiameter}. Let $\mu$ be a compactly supported probability measure on $G$ with a spectral gap on $X$. 
		
		Then there are $\varepsilon_0 = \varepsilon_0(\mu,c_1,c_2) > 0$ and $\theta=\theta(\mu)$ such that for every $(c_1, c_2 , \varepsilon)$-Diophantine probability measure $\mu_D$ with $0<\varepsilon \leq \varepsilon_0$ the following holds: There exists $\beta = \beta(\mu,\eps)$ such that for every bounded Lipschitz function $f \in \mathrm{Lip}(X)$, $x_0 \in X$ and $n\geq 1$, $$\int f(g x_0) \, d(\mu^{*n}*\mu_D^{*\beta \cdot  n})(g) = \int f \, d\Haarof{X} + O_{\Lambda, \mu, \mu_D}\big((\mathrm{Lip}(f) + \operatorname{ht}(x_0)\|f\|_{\infty}\big) e^{- \theta n}).$$
		Moreover, one may choose $\theta\ge E_7\min(R(\mu),\mathrm{gap}(\mu))$, where $E_7$ is a constant depending only on $G$ and $\Lambda$.
	\end{theorem}

	In order to apply \eqref{exponentialmixing}, we next prove that we can compare $\int f \, d\nu_n$ with a suitable inner product. 
	
	\begin{lemma}\label{MeasureApproximation}
		Let $\mu_D$ be a probability measure on $G$ with Zariski dense support. Let $x_0 \in X$ and $\nu_n = \mu_D^{*n}*\delta_{x_0}$. For $\delta>0$, $0<\eta < 1$ and $n\in\mathbb{N}$, let $h_{n,\delta,\eta}:X\to\mathbb{R}_{\ge0}$ be the function defined by \begin{equation}\label{hDefi}
			h_{n,\delta,\eta}(x):=\frac{1}{m_G(B_\delta)}\nu_n(B_\delta^X(x)) 1_{\{ \operatorname{ht} \leq E_1\delta^{-\kappa_1\eta} \}}(x).
		\end{equation} Then for any $f\in \mathrm{Lip}(X)$ and $\delta > 0$ small enough,
		\begin{align}
			\int_X f\cdot  1_{\{ \operatorname{ht} \leq E_8\delta^{-\kappa_1\eta}\}} \, d\nu_n&=\int_X f(x)h_{n,\delta,\eta}(x)\, dm_X(x) \nonumber \\ &+O_{\Lambda, \mu_D}(\delta \cdot \mathrm{Lip}(f)+  \mathrm{ht}(x_0)\delta^{\kappa_1\eta}\|f\|_{\infty}), \label{hEstimate}
		\end{align}
		where $E_8:=E_2^{-1}E_1$.
	\end{lemma}
	
	\begin{proof}
		Note that if $\operatorname{ht}(x) \leq E_1\delta^{-\kappa_1\eta}$ then $\operatorname{inj}(x)\geq \delta^\eta \geq 2\delta$ by (1) of Proposition \ref{heightftn} for $\delta$ small enough in terms of $\eta$. Therefore, using Fubini's Theorem, 
		
		\begin{align*}
			&\int_X f(y)h_{n,\delta,\eta}(y)\, dm_X(y)  \\ &= \frac{1}{m_G(B_\delta)} \int_X  \int_X f(y)1_{B_\delta^X(y)}(x) 1_{\{ \operatorname{ht} \leq E_1\delta^{-\kappa_1\eta} \}}(y)\, d\nu_n(x)dm_X(y) \\
			&=\int_X \left(\frac{1}{m_G(B_\delta)}\int_{B_{\delta}} f(gx) 1_{\{ \operatorname{ht} \leq E_1\delta^{-\kappa_1\eta} \}}(gx) \, dm_G(g) \right) d\nu_n(x). 
		\end{align*}
		For convenience, write  $$A_{\delta,\eta}(x) = \frac{1}{m_G(B_\delta)}\int_{B_{\delta}} f(gx) 1_{\{ \operatorname{ht} \leq E_1\delta^{-\kappa_1\eta} \}}(gx) \, dm_G(g).$$       
		
		If $\operatorname{ht}(x) \leq E_8\delta^{-\kappa_1\eta}$, then $\operatorname{ht}(gx) \leq E_1\delta^{-\kappa_1\eta}$ for any $g\in B_\delta$ by (3) of Proposition \ref{heightftn}. We thus have for $x$ with $\operatorname{ht}(x) \leq E_8\delta^{-\kappa_1\eta}$,
		$$f(gx) 1_{\{ \operatorname{ht} \leq E_1\delta^{-\kappa_1\eta} \}}(gx)= f(x)1_{\{ \operatorname{ht} \leq E_8\delta^{-\kappa_1\eta} \}}(x)+O(\delta\cdot\operatorname{Lip}(f)),$$
		implying
		\begin{equation}\label{hEstimateCpt}
			\int_{\{\operatorname{ht} \leq E_8\delta^{-\kappa_1\eta}\}} A_{\delta,\eta}(x) \, d\nu_n(x) =\int_X f\cdot  1_{\{ \operatorname{ht} \leq E_8\delta^{-\kappa_1\eta}\}}d\nu_n(x)+O(\delta\cdot\operatorname{Lip}(f))
		\end{equation}
		On the other hand, as $||A_{\delta,\eta}||_{\infty} \leq ||f||_{\infty}$, by Lemma \ref{QnD},
		\begin{align}
			\int_{\{\operatorname{ht}> E_8\delta^{-\kappa_1\eta}\}} A_{\delta,\eta}(x) \, d\nu_n(x) &\leq \nu_n(\{x\in X: \operatorname{ht}(x)> E_8\delta^{-\kappa_1\eta}\})\|f\|_{\infty} \nonumber  \\ &\ll_{\Lambda, \mu_D}   \mathrm{ht}(x_0)\delta^{\kappa_1\eta}\|f\|_{\infty}. \label{hEstimateNcpt}
		\end{align}
		Combining \eqref{hEstimateCpt} and \eqref{hEstimateNcpt} we get \eqref{hEstimate}.
	\end{proof}
	
	\begin{proof}[Proof of Theorem~\ref{HighDimensiontoQuantitativeEquidistribution0}]
		Upon replacing $f$ by $f - \int f \, d\Haarof{X}$ it suffices to show the claim for a bounded Lipschitz function $f \in \mathrm{Lip}(X)$ satisfying $\int f \, d\Haarof{X} = 0$ and therefore it suffices to show $$\bigg| \int f(gx_0) \, d(\mu^{*n}*\mu_D^{*\beta \cdot n})(g) \bigg| \ll_{\Lambda, \mu_D} (\mathrm{Lip}(f) + \operatorname{ht}(x_0)||f||_{\infty}) e^{-\theta n},$$ where $\beta$ is a parameter that will be determined below and for convenience we make no notational distinction between possibly non-integer number $\beta \cdot n$ and the closest integer to it.
		
		Denote $F=\pi_X(\mu)^n f$ and $\nu_n = \mu_D^{*\beta \cdot n} * \delta_{x_0}$ such that 
		\begin{equation}\label{Estimate1}
			\bigg| \int f(gx_0) \, d(\mu^{*n}*\mu_D^{*\beta \cdot n})(g) \bigg| = \bigg|\int F \, d\nu_{n}\bigg|.
		\end{equation}
		
		Set 
		\begin{equation}\label{ParameterGamma}
			\gamma = \min\left(\frac{\mathrm{gap}(\mu)}{16 E_6 R(\mu)},\frac{E_4\kappa_1}{2}\right).
		\end{equation} Then by Lemma \ref{QnD} it holds that 
		\begin{equation}\label{IntegralSplit}
			\begin{aligned}
				\int F \, d\nu_n &= \int F 1_{\{\mathrm{ht} \leq  E_8\delta^{-\frac{\gamma}{E_4}}\}} \, d\nu_{n}+\int F 1_{\{\mathrm{ht} >  E_8\delta^{-\frac{\gamma}{E_4}}\}} \, d\nu_{n} \\
				&=\int F 1_{\{\mathrm{ht} \leq  E_8\delta^{-\frac{\gamma}{E_4}}\}} \, d\nu_{n} + O_{\Lambda,\mu_D}( \operatorname{ht}(x_0)\delta^{\frac{\gamma}{E_4}}\|F\|_{\infty}).
			\end{aligned}
		\end{equation}
		
		Let $\eta=\frac{\gamma}{E_4\kappa_1}$ so that $0 < \eta \leq \frac{1}{2}$ and $\frac{\gamma}{E_4}=\kappa_1\eta$. By Lemma \ref{MeasureApproximation} and \eqref{IntegralSplit} we have
		$$\int F \, d\nu_{n} =\int_X F(x)h_{n,\delta,\eta}(x)\, dm_X(x)+O_{\Lambda,\mu_D}(\delta \cdot \mathrm{Lip}(F)+\mathrm{ht}(x_0)\delta^{\frac{\gamma}{E_4}}\|F\|_{\infty}).$$
		Recall that $\mathrm{supp}(\mu) \subset B_{R(\mu)}$. Then for some absolute constant $E_6>0$,
		$$\operatorname{Lip}(F)\leq \left(\sup_{g\in\mathrm{supp}(\mu^{*n})}\|g\|\right)\operatorname{Lip}(f) \leq e^{E_6R(\mu)n}\operatorname{Lip}(f),\quad \|F\|_{\infty}\leq \|f\|_{\infty}.$$
		Hence,
		\begin{equation}\label{Estimate2}
			\int F \, d\nu_{n} =\int_X F(x)h_{n,\delta,\eta}(x)\, dm_X(x)+O_{\Lambda,\mu_D}(\delta e^{E_6R(\mu)n} \mathrm{Lip}(f)+\mathrm{ht}(x_0)\delta^{\frac{\gamma}{E_4}}\|f\|_{\infty}).
		\end{equation}
		
		Let $\varepsilon_0'(c_1,c_2,\gamma)$ and $C_0'(c_1,c_2,\gamma)$ be the constants from  Proposition \ref{HighDimension}. We write $C_0 = C_0(c_1,c_2,\gamma) = C_0'(c_1,c_2,\gamma)$ and$$\varepsilon_0(c_1,c_2,\gamma):=\min(\varepsilon'_0(c_1,c_2,\gamma),e^{-2E_6C_0R(\mu)}).$$  Let $\mu_D$ be a $(c_1,c_2,\eps)$-Diophantine probability measure with $\varepsilon\leq\varepsilon_0(c_1,c_2,\gamma)$ and set $$\beta = \beta(\mu,\eps) =  \frac{2C_0 E_6 R(\mu)}{ \log \frac{1}{\eps}}.$$ Then we claim that for $n\asymp \frac{1}{2 E_6 R(\mu)}\log \frac{1}{\delta}$ it holds that $\|h_{n,\delta,\eta}\|_{2}\ll \delta^{-2\gamma}$. Indeed, with this choice of $n$ it holds that $$\beta \cdot n \asymp C_0 \frac{\log \frac{1}{\delta}}{\log \frac{1}{\eps}} $$ and hence by Proposition \ref{HighDimension}, for all $x \in X$, $$\nu_n(B_{\delta}^X(x)) \leq \mathrm{ht}(x)^{E_4}\delta^{\dim G - \gamma}\quad \textrm{as}\quad \beta \cdot n \asymp C_0\frac{\log \frac{1}{\delta}}{\log \frac{1}{\varepsilon}}.$$ Therefore $||h_{n, \delta , \eta}||_{\infty} \ll \delta^{-2\gamma}$ for $n \asymp \frac{1}{2 E_6 R(\mu)} \log \frac{1}{\delta}$, hence $$\|h_{n,\delta,\eta}\|_{2}^2\leq \|h_{n,\delta,\eta}\|_{L^{\infty}}^2 \ll \delta^{-4\gamma}.$$
		Applying \eqref{exponentialmixing} with $f_1=f$ and $f_2=h_{n,\delta,\eta}$,
		\begin{equation}\label{Estimate3}
			\begin{aligned}
				\bigg|\int_X F(x)h_{n,\delta,\eta}(x)\, dm_X(x)\bigg|&=\bigg|\int_X (\pi_X(\mu)^n f)(x)h_{n,\delta,\eta}(x)\, dm_X(x)\bigg|\\
				&\ll e^{-\mathrm{gap}(\mu) \cdot n/2}\|f\|_2\|h_{n,\delta,\eta}\|_2\\
				&\ll e^{-\mathrm{gap}(\mu) \cdot n/2}\delta^{-2\gamma}\|f\|_2
			\end{aligned}
		\end{equation} for $n$ sufficiently large. 
		
		Combining \eqref{Estimate1}, \eqref{Estimate2} and \eqref{Estimate3},
		\begin{align}
			\bigg|\int f \, d(\mu^{*n}*\nu_n)\bigg| &= \bigg|\int F \, d\nu_{n}\bigg| \nonumber \\
			&\ll_{\Lambda,\mu_D} \delta e^{E_6R(\mu)n} \mathrm{Lip}(f)+\mathrm{ht}(x_0)\delta^{\frac{\gamma}{E_4}}\|f\|_{\infty}+e^{-\mathrm{gap}(\mu) \cdot n/2}\delta^{-2\gamma}\|f\|_2. \label{MainEstimate}
		\end{align}
		Choosing $\delta=e^{-2 E_6 R(\mu) n}$ and as $\gamma = \min\left(\frac{\mathrm{gap}(\mu)}{16 E_6 R(\mu)},\frac{E_4\kappa_1}{2} \right)$, we can bound \eqref{MainEstimate} by $$\leq e^{-E_6R(\mu)n}\mathrm{Lip}(f) + h(x_0)e^{-\min(\frac{\mathrm{gap}(\mu)}{4 E_4},E_6\kappa_1 R(\mu))n}||f||_{\infty} + e^{-\frac{\mathrm{gap}(\mu)}{4}n}||f||_2$$ for sufficiently large $n$. Therefore setting $$\theta = \min\left(E_6R(\mu), E_6\kappa_1R(\mu) ,\frac{\mathrm{gap}(\mu)}{4 E_4},\frac{\mathrm{gap}(\mu)}{4} \right) \geq E_7\min(R(\mu),\mathrm{gap}(\mu))$$ for $E_7$ an absolute constant depending on $G$ and $\Lambda$ we conclude
		\begin{equation}
			\begin{aligned}
				\bigg|\int f(gx_0) \, d(\mu^{*n}*\mu_D^{*\beta \cdot n})(g)\bigg|&\ll_{\Lambda,\mu_D} e^{-\theta n}(\mathrm{Lip}(f) + \operatorname{ht}(x_0)||f||_{\infty}).
			\end{aligned}
		\end{equation} for $n$ large enough in terms of of $E_6$ and $R(\mu)$ and therefore $\delta$ small enough. Thus to get a bound holding for all $n$, the implied constant additionally depends on $\mu$.
	\end{proof}
	
	We furthermore mention the following corollary of our method concerning quantitative equidistribution of $\mu^{*n}*\delta_{x_0}$. We note that the assumption~\eqref{HighDimAssumption} is satisfied by \cite{BenoistDeSaxce2016} for compact simple Lie groups and $\mu$ a symmetric measure supported on finitely many matrices with algebraic entries and generating a dense subgroup. For the latter case, we mention that the Lipschitz norm appears in \eqref{EffectiveEqui} instead of the Sobolev norm, as is common in the literature.  
	
	\begin{corollary}\label{HighDimensiontoQuantitativeEquidistribution}
		Let $G$, $\Lambda$ and $X$ be as in Theorem~\ref{EffectiveDiameter}. Let $\mu$ be a compactly supported probability measure on $G$ with a spectral gap on $X$, let $x_0 \in X$ and denote $\nu_n = \mu^{*n} *\delta_{x_0}$. 
		
		Then there exists $\gamma = \gamma(\mu)$ such that the following holds. Assume there exists constants $C_0, E_4 >0$ such that for any $\delta > 0$ small enough, $x\in X$, and $n \asymp C_0 \log \frac{1}{\delta}$ \begin{equation}\label{HighDimAssumption}
			\nu_n(B_{\delta}^X(x)) \leq \mathrm{ht}(x)^{E_4}\delta^{\dim G - \gamma}.
		\end{equation} Then there exists $\theta = \theta(\mu)$ such that  for every bounded Lipschitz function $f \in \mathrm{Lip}(X)$ and $n\geq 1$, 
		\begin{equation}\label{EffectiveEqui}
			\int f(gx_0)  \, d\mu^{*n}(g) = \int f \, d\Haarof{X} + O_{\Lambda,\mu, C_0, E_4}((\mathrm{Lip}(f) + \operatorname{ht}(x_0)||f||_{\infty}) e^{- \theta n}).
		\end{equation}
	\end{corollary}
	
	\begin{proof}
		The proof is similar to the one of Theorem~\ref{HighDimensiontoQuantitativeEquidistribution0}. Indeed, we write $n = m + n_0$ for $m =  \alpha \cdot  n $ and $n_0 = \beta \cdot n$ with $\alpha,\beta > 0$ fixed constants to be determined and satisfying $\alpha + \beta = 1$. For a bounded Lipschitz function $f\in L^2_0(X)$ denote $F = \pi_X(\mu)^mf$ and then $\int f(gx_0)  \, d\mu^{*n}(g) = \int F(x) \, d\nu_{n_0}(x)$. One then proceeds as in the proof of \eqref{MainEstimate} to deduce that $| \int F \, d\nu_{n_0} |$ can be bounded for any $\gamma > 0$, for which \eqref{HighDimAssumption} holds, by 
		\begin{equation}\label{MainEstimateNew}
			\ll_{\Lambda,\mu} \delta e^{E_6 R(\mu) m}\mathrm{Lip}(f) + \mathrm{ht}(x_0)\delta^{\frac{\gamma}{E_4}}||f||_{\infty} + e^{-\mathrm{gap}(\mu)\cdot m/2}\delta^{-2\gamma}||f||_2,
		\end{equation}
		assuming that $n_0 \asymp C_0 \log \frac{1}{\delta}$.
		
		We proceed with choosing a suitable $\gamma$. Indeed if $\delta = e^{-A n}$ for $A > 0$ a constant to be chosen, then \eqref{MainEstimateNew} is bounded by 
		\begin{equation}\label{MainEstimateChooseConstant}
			\ll_{\Lambda,\mu} e^{E_6 R(\mu) m - An}\mathrm{Lip}(f) + \mathrm{ht}(x_0)e^{-A\frac{\gamma}{E_4}n}||f||_{\infty} + e^{2\gamma A n-\mathrm{gap}(\mu)\cdot m/2}||f||_2
		\end{equation} assuming that $\beta n = n_0 \asymp A \cdot  C_0 \cdot  n$ or equivalently $1 \asymp \frac{A\cdot C_0}{\beta}$. 
		
		Therefore, in order for \eqref{MainEstimateChooseConstant} to decay as claimed in \eqref{EffectiveEqui}, we require $1 \asymp \frac{A\cdot C_0}{\beta}$ as well as $$E_6 R(\mu) m - An < 0 \quad \text{ and } \quad 2\gamma A n-\mathrm{gap}(\mu)m/2 < 0,$$ which is equivalent to $$ \frac{E_6 R(\mu)}{A} < \frac{n}{m} = \frac{1}{\alpha} < \frac{\mathrm{gap}(\mu)}{4\gamma A } \quad\text{ and }\quad 1 \asymp \frac{A \cdot C_0}{\beta}.$$ Choosing $\gamma, \alpha$ and $A$ suitably the claim follows. Indeed, we may choose the parameter $\gamma$ as in the proof of Theorem~\ref{HighDimensiontoQuantitativeEquidistribution0} as  $$\gamma = \min\left(\frac{\mathrm{gap}(\mu)}{16 E_6 R(\mu)},\frac{E_4\kappa_1}{2}  \right).$$ Further we choose $\alpha > 0$ small enough such that $\frac{\alpha}{1 - \alpha} < \frac{E_6 R(\mu)}{C_0}$ and finally $A = 2E_6 R(\mu)\alpha$. With these choices, \eqref{EffectiveEqui} holds for a suitable $\theta$. 
	\end{proof}
	
	\section{Proof of Theorem~\ref{EffectiveDiameter} and Theorem~\ref{EffectiveDensity}}\label{SectionEffectiveDensity}
	
	\subsection{Proof of Theorem~\ref{EffectiveDiameter}}
	By \cite{BreuillardGelander2003}, a dense subgroup of $G$ contains a finitely generated dense subgroup, so we may assume that $S$ is a finite set. Let $\mu$ be the uniform probability measure on the finite symmetric set $S\subset G$. We distinguish the case when $G$ is compact and non-compact. If $G$ is compact, then by \cite{BenoistDeSaxce2016}, \eqref{HighDimAssumption} is satisfied and hence \eqref{EffectiveEqui} holds. For compact $G$, the latter straightforwardly implies the conclusion of Theorem~\ref{EffectiveDiameter}.

	We assume for the remainder of the proof that $G$ is non-compact. By Lemma~\ref{muspectralgap}, since $\mu$ generates a dense subgroup of $G$ and since $G$ is non-ameanable, $\mu$ has a spectral gap on $X$.  Let $c_1,c_2>0$ be the constants from Theorem \ref{BIG17Theorem3.1} for $\Gamma=\langle S\rangle$, and let $\varepsilon_0=\varepsilon_0(\mu,c_1,c_2)$ and $\theta=\theta(\mu)$ be as in Theorem \ref{MainTheorem}. By Theorem \ref{BIG17Theorem3.1} there exists a finitely supported symmetric $(c_1,c_2,\varepsilon)$-Diophantine probability measure $\mu_D$ with $0<\varepsilon\leq\varepsilon_0$ satisfying $\operatorname{supp}(\mu_D)\subset \Gamma\cap B_{\varepsilon}(e)$. Let $\beta = \beta(\mu,\eps)$ be the constant from Theorem~\ref{MainTheorem}. Since $\mu_D$ is finitely supported and its support is contained in $\Gamma$, there exists an integer $k_0\in\mathbb{N}$ such that $\operatorname{supp}(\mu_D^{*\beta}) \subseteq\operatorname{supp}(\mu^{*k_0}) = S^{k_0}$.
	
	By Theorem \ref{MainTheorem}, for any $f \in \mathrm{Lip}(X)$ and $n\geq 1$, \begin{equation}\label{EffectiveEquidistributionEstimate}
		\int f\, d(\mu^{*n}*\mu_D^{*\beta \cdot n}*\delta_{x_0}) = \int f \, d\Haarof{X} + O_{\Lambda,\mu}((\mathrm{Lip}(f) + \operatorname{ht}(x_0)||f||_{\infty}) e^{- \theta n}).
	\end{equation}
	Using (1) of Proposition \ref{heightftn}, for $r$ sufficiently small, $\mathrm{inj}(y)>r^{\frac{2}{\kappa_1}}$ for $y\in X(r^{-1})$. Recall $0<\kappa_1<1$ and let us write $r_1=r^{\frac{2}{\kappa_1}}$ for simplicity. For each $y\in X(r^{-1})$ we may choose a bounded Lipschitz function $f_{r,y}\in\operatorname{Lip}(X)$ such that $1_{B_{\frac{1}{2}r_1}(y)}\leq f_{r,y}\leq 1_{B_{r_1}(y)}$, $\operatorname{Lip}(f_{r,y})\ll r_1^{-1}$, and $\int f_{r,y} \, dm_X \gg r_1^{\dim G}$, where the implied constants only depend on $G$. Choose
	$$n= \left\lceil\frac{4\dim G}{\kappa_1\theta}(\log r^{-1}+\log \operatorname{ht}(x_0))\right\rceil,$$
	then the error term in \eqref{EffectiveEquidistributionEstimate} is bounded by
	\begin{align*}
		O_{\Lambda,\mu}((\mathrm{Lip}(f) + \operatorname{ht}(x_0)||f||_{\infty}) e^{- \theta n}) \ll_{\Lambda,\mu} (r_1^{-1}+\operatorname{ht}(x_0))e^{- \theta n} \ll_{\Lambda,\mu}  r_1^{2\dim G - 1}.
	\end{align*}
	It follows that
	$$\int f_{r,y}\, d(\mu^{*n}*\mu_D^{*\beta \cdot n}*\delta_{x_0}) \gg  r_1^{\dim G}  + O_{\Lambda,\mu}(r_1^{2\dim G - 1})>0$$
	for sufficiently small $r$ depending on $\Lambda$ and $\mu$. Therefore for any $y\in X(r^{-1})$ there exists $x\in \operatorname{supp}(f_{r,y})\subset B_r(y)$ with $$x\in \operatorname{supp}(\mu^{*n}*\mu_D^{*\beta \cdot n}*\delta_{x_0})\subseteq \operatorname{supp}(\mu^{*(k_0+1)n}*\delta_{x_0}) = S^{n(k_0 + 1)}x_0,$$
	hence we conclude that for $r$ sufficiently small depending on $\Lambda$ and $\mu$,
	$$\mathrm{diam}_{r}(X,S,x_0) \leq \frac{4\dim G}{\kappa_1\theta}(k_0+1)(\log r^{-1}+\log \operatorname{ht}(x_0)).$$
	The proof of Theorem~\ref{EffectiveDiameter} is complete. 
	
	\subsection{Proof of Theorem~\ref{EffectiveDensity}}
	For any given $r>0$ we may choose a maximal $r$-separated subset $\Xi$ of $\{\operatorname{ht}\leq r^{-1}\}$. Let $\Xi=\{x_1,\ldots,x_{|\Xi|}\}$. Note that $|\Xi|\asymp r^{-\dim G}$ and $X(r^{-1})$ is covered by the balls $B^X_r(x_1),\ldots,B^X_r(x_{|\Xi|})$.  
	
	Throughout the proof we fix a constant $A \geq 2 \dim G$. By Theorem~\ref{EffectiveDiameter}, there is a constant $C$ depending on $\Lambda$ and $S$ such that for $N = \lceil C\log r^{-2A}\rceil$ and any $y \in X(r^{-2A})$ there is $h(y,N,i) \in S^N$ such that $h(y,N,i)y \in B_{r}(x_i)$. We use here that $y \in X(r^{-2A})$ and therefore the contribution of the height of $y$ in Theorem~\ref{EffectiveDiameter} is at most of size $\log r^{-2A}$. 
	
	We claim for any $y \in X(r^{-2A})$, 
	\begin{equation}\label{MassEstimate}
		\mu^{*N}(h(y,N,i)) \geq r^B
	\end{equation}
	for $B = B(\Lambda,\mu,A)$ a constant depending on $\Lambda$, $\mu$ and $A$. Indeed, if $S$ is finite, then if every atom of $\mu$ has mass at least $0 < \rho < 1$, it follows that $\mu^{*N}(h(y,N,i)) \geq \rho^N|S|^{-N} \geq r^B$ for $B = 4\cdot A\cdot C\log(\frac{|S|}{\rho})$. If $S$ is not finite, then it must be countable as $\mathrm{Ad}(\mathrm{supp}(\mu))$ has algebraic entries. Therefore, again by \cite{BreuillardGelander2003}, there is a finite $S' \subset S$ with $\mu(S') \geq \frac{1}{2}$ and such that $S'$ generates a dense subgroup of $G$. \eqref{MassEstimate} then follows by applying the above argument to $S'$.
	
	For $x_0 \in X$ and $k\in\mathbb{N}$ denote by $P_{k,i}$ the set of elements $(g_1,\ldots,g_{kN})\in\operatorname{Supp}(\mu^{\otimes kN})$ such that the set $\{g_j\cdots g_1x_0: 1\leq j\leq kN\}$ does not intersect $B^X_r(x_i)$. 
	
	\begin{lemma}\label{PkiInductionEstimate}
		Let $A \geq 2 \dim G$ and let $1 \leq i \leq |\Xi|$. For $r$ small enough in terms of $\Lambda$ and $\mu$ the following holds. If $\mu^{\otimes kN}(P_{k,i})\geq \mathrm{ht}(x_0) \cdot r^{A + \dim G}$ for some $k \geq 1$, then $$\mu^{\otimes (k+1)N}(P_{k+1,i})\leq \left(1-r^{B}/2\right)\mu^{\otimes kN}(P_{k,i}).$$
	\end{lemma}
	
	\begin{proof}
		Let $Q_k$ be the set of elements $(g_1,\ldots,g_{kN})\in\operatorname{Supp}(\mu^{\otimes kN})$ such that $$g_{kN}\cdots g_1x_0\notin X(r^{-(A + 2 \dim G)}).$$ By Lemma \ref{QnD},
		\begin{align*}
			\mu^{\otimes kN}(Q_k)&=\mu^{*kN}(\{g\in G: \operatorname{ht}(gx_0)\ge r^{-(A + 2 \dim G)} \}) \\ &\ll_{\Lambda,\mu} r^{A + 2 \dim G}\operatorname{ht}(x_0).
		\end{align*}
		If $(g_1,\ldots,g_{kN})\in \operatorname{Supp}(\mu^{\otimes kN})\setminus Q_k$, then writing $y=g_{kN}\cdots g_1x_0$, there exists $h(y,N,i)\in S^N$ such that $h(y,N,i)g_{kN}\cdots g_{1}x_0\in B^X_r (x_i)$. It follows that
		\begin{equation*}
			\begin{aligned}
				\mu^{\otimes (k+1)N}(P_{k+1,i})&\leq \mu^{\otimes kN}(P_{k,i})-\mu^{\otimes kN}(P_{k,i}\setminus Q_k)\inf_{y \in X(r^{-2A})}\mu^{* N}(h(y,N,i))\\
				&\leq \mu^{\otimes kN}(P_{k,i})-r^{B}\mu^{\otimes kN}(P_{k,i}\setminus Q_k).
			\end{aligned}
		\end{equation*}
		Since we assume $\mu^{\otimes kN}(P_{k,i})\geq \mathrm{ht}(x_0)r^{A + \dim G}$, choosing $r$ sufficiently small in terms of $\mu$, 
		$$\mu^{\otimes kN}(P_{k,i}\setminus Q_k)\geq \mu^{\otimes kN}(P_{k,i})-\mu^{\otimes kN}(Q_{k})\geq \frac{1}{2}\mu^{\otimes kN}(P_{k,i}),$$
		showing the claim.
	\end{proof}

	\begin{proof}[Proof of Theorem~\ref{EffectiveDensity}]
		Let $A \geq 2\dim G$ and let $M_r = \lceil 16 r^{-D \cdot A}\rceil$ for $D$ a constant to be chosen sufficiently large depending on $\Lambda$ and $\mu$. Since $\log(1-x)\leq -x/4$ (for $x$ small enough), we have for sufficiently small $r$,
		$$M_r\log\left(1-\frac{1}{2}r^{B}\right)\leq -\frac{1}{8}r^{B}M_r\leq -r^{-A} \leq -r^{-A/2} + \dim G \cdot \log r,$$ choosing $D$ sufficiently large such that $B - DA \leq -A$. Notice that we may choose $D$ to only depend on $\Lambda$ and $\mu$.
		
		Hence $\left(1-\frac{1}{2}r^{B}\right)^{M_r} \leq  \exp(-r^{-A/2})r^{\dim G}$. Iterating Lemma~\ref{PkiInductionEstimate}, we deduce that 
		$$\mu^{\otimes M_r N}(P_{M_r,i})\leq \mathrm{ht}(x_0) \cdot r^{A + \dim G}$$
		for any $1\leq i\leq |\Xi|$. Let $\mathcal{E}:=\displaystyle\bigcup_{i=1}^{|\Xi|}P_{M_r, i}$. Then we have 
		$$\mu^{\otimes M_r N}(\mathcal{E})\leq \mathrm{ht}(x_0) \cdot  r^{A + \dim G} \cdot |\Xi|\ll \mathrm{ht}(x_0) \cdot r^{A},$$ and for any $(g_1,\cdots,g_{M_r N})\in \operatorname{Supp}(\mu^{\otimes M_r N})\setminus \mathcal{E}$ and $1\leq i\leq |\Xi|$ there exists $1\leq j(i)\leq M_r N$ such that $g_{j(i)}\cdots g_1x_0\in B^X_r(x_i)$, i.e. the set $\{g_j\cdots g_1x_0: 1\leq j\leq M_r N\}$ is $r$-dense in $X(r^{-1})$. The proof of \eqref{ProbEstimateEffectiveDensity} is complete as $M_{r}N \leq r^{-2 \cdot  D\cdot A}$ for $r$ small enough in terms of $\Lambda$, $\mu$ and $A$ upon replacing $A$ by $\frac{A}{2\cdot D}$ and thus setting $\alpha = \frac{1}{2 \cdot D}$. With this replacement, for the above argument to apply we require that $\frac{A}{2D} \geq 2 \dim G$.
		
		To pass from \eqref{ProbEstimateEffectiveDensity} to an almost sure statement about the orbit $(Y_{n,x_0})_{n \geq 1}$ we apply the Borel-Cantelli Lemma for $A > 0$ large enough. Indeed, for $n > 0$ write $r_n = n^{-\frac{1}{A}}$  and consider the set $$F_{n} = \{ (g_1, g_2, \ldots) \in G^{\mathbb{N}} \,:\, g_{n}\cdots g_1x_0 \text{ is not } r_n \text{-dense in } X(r_n^{-1}) \}.$$ Then $\sum_{n \geq 1} \mathbb{P}[F_n] < \infty$ for $A$ large enough and hence by the Borel-Cantelli Lemma the claim follows. 
	\end{proof}
	
	We note that in the cocompact case, we can improve the error rate in Theorem~\ref{EffectiveDensity} to an exponential one.
	
	\begin{theorem}\label{CompactEffectiveDensity}
		Let $G$, $\Lambda$, $X$ and $S$ be as in Theorem~\ref{EffectiveDiameter} and let $\mu$ be a probability measure on $G$ with support $S$. Assume that $X$ is compact. Then for $A > 0$ large enough depending on $\Lambda$ and $\mu$ the following holds: For any $x_0 \in X$,  
		\begin{equation}
			\mathbb{P}[(Y_{1,x_0}, \ldots , Y_{\lceil r^{-A}\rceil,x_0}) \text{ is not } r\text{-dense in } X(r^{-1})] \leq \mathrm{ht}(x_0)\cdot \exp(-r^{-\alpha \cdot  A}),
		\end{equation} for $r$ small enough in terms of $\Lambda, \mu$ and $A$ and for $\alpha = \alpha(\Lambda,\mu) > 0$ a constant depending on $\Lambda$  and $\mu$.
	\end{theorem}
	
	\begin{proof}
		The proof is as the one of Theorem~\ref{EffectiveDensity} without requiring to deal with quantitative non-divergence. Indeed, we may drop the lower bound on $\mu^{\otimes kN}(P_{k,i})$ in Lemma~\ref{PkiInductionEstimate} and proceed otherwise as in the proof of Theorem~\ref{EffectiveDensity}.
	\end{proof}
	
	\bibliography{references.bib}
\end{document}